\documentclass[12pt,reqno]{amsart}

\usepackage{amsfonts,amsmath,amssymb,amsthm}
\usepackage{enumitem}
\usepackage{amsaddr}
\usepackage{mathrsfs}
\usepackage{mathabx}
\usepackage{bbm}
\usepackage{centernot}
\usepackage{mathtools}
\usepackage{graphicx}

\newtheorem{Theorem}{Theorem}[section]

\newtheorem{Proposition}[Theorem]{Proposition}

\newtheorem{Definition}[Theorem]{Definition}
		
\numberwithin{equation}{section}
\usepackage[a4paper,margin=3cm]{geometry}
\begin{document}
\sloppy
\title[A convergence not metrizable]
{A convergence not metrizable}

\author[L. Rivera]{Luis David Rivera Barreno}
\address{\vspace{-8mm} Grupo de Geometría EC\\ 
    	Universidad Central del Ecuador (UCE)\\
    	Quito, Ecuador}
\email{ldriverab18@gmail.com}

\date{December 19, 2025} 

\begin{abstract}
    Certain notions of convergence of sequences functions such
    as pointwise convergence and (uniform) convergence on compact 
    or bounded sets come from suitable topological function spaces; see \cite{Lima}. 
    Under certain conditions these topologies involved
    are metrizable, which in an advantage since there is an extensive theory on
    convergence in metric spaces. However, the case of pointwise convergence is
    delicate, since it is shown that under certain hypotheses this form of convergence 
    of sequences of functions is not equivalent to convergence in metric.
\end{abstract}


\maketitle
%%%%%%%%%%%%%%%%%%%%%%%%%%%%%%%%%%%%%%%%%%%%%%%%%%%%%%%%%%%%%%%%%%%%%%%%%%%%%%%%%%%%%%%%
%%%%%%%%%%%%%%%%%%%%%%%%%%%%%%%%%%%%%%%%%%%%%%%%%%%%%%%%%%%%%%%%%%%%%%%%%%%%%%%%%%%%%%%%
In the paper \cite{Fort}, M. K. Fort Jr., establishes the need to provide a 
correct proof of the fact that there is no metric on the set of continuous real
functions defined on the closed interval $[0,1]$ such that the convergence 
of a sequence of functions with respect to this metric is equivalent 
to its pointwise convergence. This paper aims to provide a slightly more 
general proof, but using the ideas of \cite{Lima}.

%%%%%%%%%%%%%%%%%%%%%%%%%%%%%%%%%%%%%%%%%%%%%%%%%%%%%%%%%%%%%%%%%%%%%%%%%%%%%%%%%%%%%%%%
%%%%%%%%%%%%%%%%%%%%%%%%%%%%%%%%%%%%%%%%%%%%%%%%%%%%%%%%%%%%%%%%%%%%%%%%%%%%%%%%%%%%%%%%
Denote by the set of natural numbers to $\mathbb{N}:=\{1,2,3,\ldots\}$. 
The symbols $\mathbb{R}$, $\mathbb{Q}$ and 
$\mathbb{R}\setminus \mathbb{Q}$ denote the set of real numbers, 
the set of rational numbers and
the set of irrational numbers, respectively. 
For given metric spaces $(M,d_{M})$ and $(N,d_{N})$, 
the set of functions from $M$ to $N$ is denoted as $\mathcal{F}(M,N)$ 
and the set of continuous functions from $M$ to $N$ is denoted as $C(M,N)$.

%%%%%%%%%%%%%%%%%%%%%%%%%%%%%%%%%%%%%%%%%%%%%%%%%%%%%%%%%%%%%%%%%%%%%%%%%%%%%%%%%%%%%%%%
%%%%%%%%%%%%%%%%%%%%%%%%%%%%%%%%%%%%%%%%%%%%%%%%%%%%%%%%%%%%%%%%%%%%%%%%%%%%%%%%%%%%%%%%

In order to fix terminology, we consider the following definition.

\begin{Definition}
    A metric space $(M,d)$ is said to be \emph{strongly 
    second contable} if there exists a dense and contable subset $D\subset M$ 
    such that $M\setminus D$ is also dense. 
\end{Definition}

An immediate example of a \emph{strongly second contable} metric space is
$\mathbb{R}$ endowed with the usual topology, since $\mathbb{Q}$ is a dense contable 
subset whose complement $\mathbb{R}\setminus \mathbb{Q}$ is dense. 

%%%%%%%%%%%%%%%%%%%%%%%%%%%%%%%%%%%%%%%%%%%%%%%%%%%%%%%%%%%%%%%%%%%%%%%%%%%%%%%%%%%%%%%
%%%%%%%%%%%%%%%%%%%%%%%%%%%%%%%%%%%%%%%%%%%%%%%%%%%%%%%%%%%%%%%%%%%%%%%%%%%%%%%%%%%%%%%

\begin{Proposition}
    Let $(M,d_{M})$ be a complete, nonempty, strongly second contable metric
    space and let $(N,d_{N})$ be a metric space having a non-unitary path-component.
    Then, there does not exist a metric $d$ on $\mathcal{F}(M,N)$ 
    such that the convergence of sequences in the space $(\mathcal{F}(M,N),d)$
    is equivalent to the pointwise convergence of sequences of functions in $\mathcal{F}(M,N)$. 
\end{Proposition}

\begin{proof}
    By reductio ad absurdum, it is assumed that there exists a metric
    $d$ on $\mathcal{F}(M,N)$ such that the convergence of sequences
    in the space $(\mathcal{F}(M,N),d)$ is equivalent to the pointwise
    convergence of sequences of functions in $\mathcal{F}(M,N)$.
    Since $(M,d_{M})$ is strongly second contable, there exists 
    a dense contable subset $D\subset M$ such that $M\setminus D$
    is dense. Let $a, b\in N$, with $a\neq b$, 
    and a continuous path $\psi:[0,1]\rightarrow N$ 
    where $\psi(0)=b$ y $\psi(1)=a$. 
    We consider the function
    \begin{eqnarray*}
        \varphi: M&\longrightarrow& N\\
        x &\longmapsto& \varphi(x):=
        \begin{cases}
            a, &\text{ if } x\in D,\\
            b, &\text{ if } x\in M\setminus D. 
        \end{cases}
    \end{eqnarray*}
    It follows that $\varphi:M\longrightarrow N$ is discontinuous in $M$.
    Let $D=\{x_{1},x_{2},\ldots,x_{n},\ldots\}$ be an enumeration of $D$. 
    For all $n\in \mathbb{N}$, let 
    \begin{eqnarray*}
        \varphi_{n}: M&\longrightarrow& N\\
        x &\longmapsto& \varphi_{n}(x):=\begin{cases}
            a, &\text{ if } x\in D_{n}:=\{x_{1},\ldots,x_{n}\},\\
            b, &\text{ if } x\in M\setminus D_{n}. 
        \end{cases}
    \end{eqnarray*}
    This sequence of functions in $\mathcal{F}(M,N)$ 
    converges pointwise to $\varphi:M\longrightarrow N$. 
    Now, for all $n\in \mathbb{N}$, 
    it will be shown that $\varphi_{n}$ can be pointwise approximated 
    by an sequence of continuous functions. In fact, let $n\in \mathbb{N}$.
    
    \begin{itemize}
        \item First, we consider $n=1$. For all $m\in \mathbb{N}$, let 
            \[
                F_{m}^{1}:=B\left[x_{1};\frac{1}{m+1}\right]\subset M 
                \,\,\, \text{ and } \,\,\, G_{m}^{1}:=M\setminus 
                B\left(x_{1};\frac{1}{m}\right)\subset M. 
            \]
            Using Urysohn's Lemma, for all $m\in \mathbb{N}$, 
            there exists a continuous function 
            $f_{U,m}^{1}:M\rightarrow [0,1]$ such that
            $f_{U,m}^{1}(F_{m}^{1})\subset \{1\}$ y 
            $f_{U,m}^{1}(G_{m}^{1})\subset \{0\}$. So, for all $m\in \mathbb{N}$, 
            we take $f_{m}^{1}:=\psi\circ f_{U,m}^{1}$.
            It follows that $(f_{m}^{1})_{m\in \mathbb{N}}$ 
            converges pointwise to $\varphi_{1}$.
        \item We now assume that $n>1$. Let $\delta>0$ be such that  
            for all $i, j\in \{1,\ldots,n\}$, if
            $i\neq j$, then 
            \[
                B(x_{i};\delta)\cap B(x_{j};\delta) = \emptyset. 
            \]
            Let $p\in \mathbb{N}$ be such that $\frac{1}{p}<\delta$. 
            For all $m\in \mathbb{N}$, we consider  
            \[
                F_{m}^{n}:=\bigcup_{i=1}^{n}B\left[x_{i};\frac{1}{m+1+p}\right]\subset
                M
            \]
            and
            \[
                G_{m}^{n}:=M\setminus \bigcup_{i=1}^{n}B\left(x_{i};\frac{1}{m+p}\right)
                \subset M.
            \]
            Since $M$ is a normal space, by Urysohn's Lemma, for all
            $m\in \mathbb{N}$, there exists a continuous function 
            $f_{U,m}^{n}:M\rightarrow [0,1]$ such that 
            $f_{U,m}^{n}(F_{m}^{n})\subset \{1\}$ y $f_{U,m}^{n}(G_{m}^{n})\subset \{0\}$. 
            Therefore, for all $m\in \mathbb{N}$, we take $f_{m}^{n}:=
            \psi\circ f_{U,m}^{n}$. Then, $(f_{m}^{n})_{m\in \mathbb{N}}$ 
            converges pointwise to $\varphi_{n}$. 
    \end{itemize}
    Thus, by hypothesis, $\varphi_{n}\in \overline{C(M,N)}$, closure of 
    $C(M,N)$ in $(\mathcal{F}(M,N),d)$. 
    It follows that $(\varphi_{m})_{m\in \mathbb{N}}$ 
    is a sequence in the space $\overline{C(M,N)}$. 
    Since
    $\varphi_{n}\xrightarrow[n\to +\infty]{} \varphi$ in $(\mathcal{F}(M,N),d)$, 
    we have $\varphi\in \overline{C(M,N)}$. So, there exists a sequence
    of continuous functions that converges to $\varphi$ 
    in $(\mathcal{F}(M,N),d)$. Then, $\varphi$ 
    can be approximated pointwise by a sequence of continuous functions.
    But by Proposition 15 of Chapter VI \cite{Lima}, 
    it follows that the set of discontinuity points of $\varphi$ is
    meagre in $M$. Since $\varphi$ is discontinuous in $M$, we have that
    $M$ is meagre in $M$. By Baire's Theorem, it follows that 
    $M=\emptyset$, which contradicts the assumption. 
\end{proof}

%%%%%%%%%%%%%%%%%%%%%%%%%%%%%%%%%%%%%%%%%%%%%%%%%%%%%%%%%%%%%%%%%%%%%%%%%%%%%%%%%%%%%%%%
%%%%%%%%%%%%%%%%%%%%%%%%%%%%%%%%%%%%%%%%%%%%%%%%%%%%%%%%%%%%%%%%%%%%%%%%%%%%%%%%%%%%%%%%

\end{document}